\input ./style/arxiv-general.cfg
\documentclass[aap,MSNbibl,seceqn,dvips]{arximspdf}
\makeatletter
   \@ifpackageloaded{graphicx}{}{\usepackage{graphicx}}
\makeatother
\usepackage{mathbh}

\doi{10.1214/15-AAP1140}% Updated by VTEXPTS2LaTeX.exe, 09.10.2015 12:54
\volume{26}
\issue{4}
\pubyear{2016}
\firstpage{2083}
\lastpage{2105}
\docsubty{FLA}

\makeatletter
\newcommand{\eqref}[1]{(\ref{#1})}
\newtheorem{thmm}{Theorem} % Theorem environment
\newtheorem{cor}{Corollary}    % Corollary environment
\newtheorem{lemma}{Lemma}        % Lemma environment
\newproclaim{remark}{Remark}
\newproclaim{assumption}{Assumption}

\def \one{\mathbh{1}}
\makeatother

\begin{document}
\begin{frontmatter}

\title{Euler approximations with varying coefficients: The~case of
superlinearly growing diffusion~coefficients}
\runtitle{Euler approximations with varying coefficients}

\begin{aug}
% Corresponding author: Sotirios Sabanis - s.sabanis@ed.ac.uk% Updated by VTEXPTS2LaTeX.exe, 12.10.2015 08:39
%by VTEXPTS2LaTeX.exe, 09.10.2015 12:54
\author{\fnms{Sotirios}~\snm{Sabanis}\corref{}\ead
[label=e1]{s.sabanis@ed.ac.uk}}
\runauthor{S. Sabanis}
\affiliation{University of Edinburgh}
\address{School of Mathematics\\
University of Edinburgh\\
Edinburgh EH9 3JZ\\
United Kingdon\\
\printead{e1}}
%
%\author[A]{\fnms{}~\snm{}\ead[label=e1]{}}%,
%\author[]{\fnms{}~\snm{}\ead[label=]{}}
% \and
%\author[]{\fnms{}~\snm{}\ead[label=]{}}

%\affiliation{}
%\dedicated{}
%address[A]{\\\printead{e1}}
%\address[]{\\\printead{}}
\end{aug}

% HISTORY:
%
\received{\smonth{12} \syear{2014}}% Updated by VTEXPTS2LaTeX.exe,
%09.10.2015 12:54
%\revised{\smonth{} \syear{}}

% ABSTRACT
%
\begin{abstract}

A new class of explicit Euler schemes, which approximate stochastic
differential equations (SDEs) with superlinearly growing drift and
diffusion coefficients, is proposed in this article. It is shown, under
very mild conditions, that these explicit schemes converge in
probability and in $\mathcal L^p$ to the solution of the corresponding
SDEs. Moreover, rate of convergence estimates are provided for
$\mathcal L^p$ and almost sure convergence. In particular, the strong
order $1/2$ is recovered in the case of uniform $\mathcal L^p$-convergence.
\end{abstract}

% KEYWORDS
\begin{keyword}[class=AMS]
\kwd[Primary ]{60H35}
%\kwd{}
\kwd[; secondary ]{65C30}
\end{keyword}
\begin{keyword}
\kwd{Explicit Euler approximations}
\kwd{rate of convergence}
\kwd{local Lipschitz condition}
\kwd{superlinearly growing coefficients}
\end{keyword}
\end{frontmatter}

%s1 #&#
\section{Introduction} \label{Intro}

Motivated by the work of \cite{Sabanis} and \cite{HJK-converge} on
explicit Euler-type schemes which approximate (in an $\mathcal L^p$
sense) SDEs with superlinearly growing drift coefficients, the author
extends the techniques developed in \cite{Sabanis} and \cite
{Gyongy-Sabanis} to obtain, under very mild assumptions, convergence
results for the case of superlinearly growing diffusion coefficients.
For an extensive and up to date literature review on Euler
approximations, one can consult \cite{HJK-converge} and \cite{HJ},
where it is demonstrated that the implementation of implicit schemes
requires significantly more computational effort than this new
generation of explicit Euler-type approximations. Thus, the focus of
this work is solely on explicit methods. For implicit methods, one
could consult \cite{HMS,Mao-Szpruch} and the references therein.

In order to highlight the progress made in this article with comparison
to the latest developments in the field, namely \cite{HJ} and \cite
{Tretyakov-Zhang}, the following example is presented; consider a
nonlinear ($d$-dimensional) SDE which is given by
\[
dX(t)=\lambda X(t) \bigl(\mu-\bigl|X(t)\bigr|\bigr)\,dt + \xi\bigl|X(t)\bigr|^{3/2}\,dW_t
\]
with initial condition $X_0\in\mathbb R^d$, where $\lambda$, $\mu$ and
all elements of the vector $X_0$ are positive constants. Moreover, $\xi
\in \mathbb R^{d\times d_1}$ is a positive definite matrix and $\{
W(t)\}_{t \geq0}$ is a $d_1$-dimensional Wiener
martingale. This SDE is chosen since its one-dimensional version is the
popular $3/2$-model in Finance (see, e.g., \cite{Goard-Mazur} and the
references therein), which is used for modelling (nonaffine)
stochastic volatility processes and for pricing VIX options. One then
further observes that the coercivity and monotonicity conditions, which
are given in {A-4} and {A-6} below, are satisfied with
$p_0 =2p_1-1$ and $p_1 = \frac{\lambda}{|\xi|^2}+1$ (for more details,
see the \hyperref[app]{Appendix}). Due to Theorem~\ref{order-thmm} below, one obtains
convergence results in $\mathcal L^2$ (or more generally in $\mathcal
L^p$) with order $1/2$ even when $p_1$ and $p_0$ are relatively small.
Consider, for example, the case $p_1=3.5$ (and thus $p_0=6$); then the
explicit Euler-type scheme in Theorem~\ref{order-thmm} below converges
to the true solution of the above SDE in $\mathcal L^2$ with order $1/2$,
whereas the authors in \cite{HJ} are able to show $\mathcal
L^p$-convergence (without rate) of their explicit schemes only for
$p<1/2$ (see Section~4.10.3 in \cite{HJ}). Also, the findings in \cite
{Tretyakov-Zhang} (see Lemma~3.1 in \cite{Tretyakov-Zhang}) do not
produce the required moment bounds for the above case, and thus, no
statement can be made about the convergence of their explicit numerical
scheme in $\mathcal L^2$.

To further highlight the advantages of the proposed approximation
methods hereunder, it is noted that Theorem~\ref{mainthm} presents
optimal $\mathcal L^p$-convergence results of explicit Euler-type
schemes under the monotonicity condition {A-3} (see below) in
the sense that $\mathcal L^p$-convergence results are obtained for any
$p<p_0$ which essentially closes the gap appearing in \cite{HJ}.
Furthermore, Theorem~\ref{uniform-order-thmm} presents \textit{uniform}
$\mathcal L^p$-convergence results with order $1/2$. The author is not
aware of any other such results for the case of explicit Euler-type
approximations to SDEs with superlinearly growing diffusion coefficients.

This section concludes by introducing some basic notation. The
norm of a vector $x \in\mathbb R^d$ and the Hilbert--Schmidt norm of
a matrix $A\in\mathbb R^{d\times m}$ are respectively denoted by
$|x|$ and $|A|$. The transpose of a matrix $A \in\mathbb R^{d\times
m}$ is denoted by $A^{T}$ and the scalar product of two vectors $x,
y \in\mathbb R^d$ is denoted by $xy$. The integer part of a
nonnegative real
number $x$ is denoted by $\lfloor x\rfloor$. Moreover, $\mathcal
L^p=\mathcal
L^p(\Omega, \mathcal F, \mathbb P)$ denotes the space of random
variables $X$ with a norm $\|X\|_p:= (\mathbb
E [|X|^p ] )^{1/p} < \infty$ for $p>0$. Finally, $\mathcal
B(V)$ denotes the $\sigma$-algebra of Borel sets of a topological
space~$V$.

%s2 #&#
\section{Main results}

Let $(\Omega, \{\mathcal F_t\}_{t \ge0}, \mathcal F, \mathbb
P)$ be a filtered probability space satisfying the usual
conditions, that is, the filtration is increasing, right continuous and
complete.
Let $\{W(t)\}_{t \geq0}$ be a $d_1$-dimensional Wiener
martingale. Furthermore, it is assumed that $b(t,x)$ and
$\sigma(t,x)$ are $\mathcal B(\mathbb R_+) \otimes
\mathcal{B}(\mathbb R^d)$-measurable functions which take values in
$\mathbb R^d$ and $\mathbb R^{d \times d_1}$, respectively. For a fixed
$T>0$, let us consider an SDE given by
%
%e2.1 #&#
\begin{equation}
\label{sde} dX(t)=b\bigl(t, X(t)\bigr)\,dt+\sigma\bigl(t,X(t)\bigr)\,dW(t)
\qquad\forall
t\in [0, T],
\end{equation}
with initial value $X(0)$ which is an almost surely finite $\mathcal
F_0$-measurable random variable.

Let constants $p_0$ and $p_1 \in[2, \infty)$. We consider the
following conditions:
\begin{longlist}[A-5.]

\item[A-1.] The function $b(t, x)$ is continuous in $x$ for
any $t\in[0, T]$.

\item[A-2.] For every $R \ge0$, there exists a constant $N_R$
such that
\[
\sup_{|x|\le R}\bigl|b(t,x)\bigr|\le N_R
\]
for any $t \in[0,T]$.

\item[A-3.] For every $R>0$, there exists a positive constant
$L_R$ such that, for any $t \in[0,T]$,
\[
2(x-y) \bigl(b(t,x)-b(t,y)\bigr) + (p_1-1)\bigl|\sigma(t,x)-
\sigma(t,y)\bigr|^2 \leq L_R |x-y|^2
\]
for all $|x|, |y| \leq R$.

\item[A-4.] There exists a positive constant $K$ such that
\[
2xb(t,x) + (p_0-1)\bigl|\sigma(t,x)\bigr|^2 \leq K
\bigl(1+|x|^2\bigr)
\]
for any $t \in[0,T]$ and $x \in\mathbb R^d$.

\item[A-5.] $\mathbb{E}[|X(0)|^{p_0}]<\infty$.
\end{longlist}

%re1 #&#
\begin{remark} \label{local_bdd_sigma}
Due to {A-2} and {A-4}, for every $R\ge0$, there exists a
constant $N^{\prime}_R$ such that
$
\sup_{|x|\le R}|\sigma(t,x)|\le N^{\prime}_R
$
for any $t \in[0,T]$.
\end{remark}

Furthermore, for every $n \geq1$, the
following numerical scheme is defined:
%
%e2.2 #&#
\begin{eqnarray}
\label{tem} dX_n(t)=b_n\bigl(t,X_n
\bigl(\kappa_n(t)\bigr)\bigr)\,dt+\sigma_n
\bigl(t,X_n\bigl(\kappa_n(t)\bigr)\bigr)\,dW(t)
\nonumber
\\[-8pt]
\\[-8pt]
\eqntext{\forall
t\in [0, T],}
\end{eqnarray}
with the same initial value $X(0)$ as equation \eqref{sde}, where
$b_n(t,x)$ and $\sigma_n(t,x)$ are
$\mathcal B(\mathbb R_+) \otimes\mathcal B(\mathbb R^d)$-measurable
functions which take values in $\mathbb R^d$ and $\mathbb R^{d \times
d_1}$, respectively, and $\kappa_n(t):=\lfloor nt\rfloor/n$. The
following conditions are considered:
\begin{longlist}[{B-1.}]
\item[B-1.] For every $R \ge0$,
%
%e2.3 #&#
\begin{eqnarray}
\label{difference} \int_0^T\sup
_{|x|\le R}\bigl[\bigl|b_n(t,x)- b(t,x)\bigr|^{p_0}+ \bigl|
\sigma_n(t,x)-\sigma (t,x)\bigr|^{p_0}\bigr]\,dt \rightarrow 0
\nonumber
\\[-8pt]
\\[-8pt]
\eqntext{\mbox{as } n\to\infty.}
\end{eqnarray}

\item[B-2.] There exist an $\alpha\in(0, 1/2]$ and a
constant $C$ such that, for every $n\ge1$,
%
%e2.4 #&#
\begin{eqnarray}
\label{uni_bound} \bigl|b_n(t,x)\bigr|&\le&\min\bigl(Cn^{\alpha}\bigl(1+|x|\bigr),
\bigl|b(t,x)\bigr|\bigr) \quad\mbox{and}
\nonumber
\\[-8pt]
\\[-8pt]
\nonumber
\bigl |\sigma_n(t,x)\bigr|^2 &\le&\min
\bigl(Cn^{\alpha}\bigl(1+|x|^2\bigr), \bigl|\sigma(t,x)\bigr|^2
\bigr),
\end{eqnarray}
for any $t \in[0,T]$ and $x \in\mathbb R^d$.

\item[B-3.] There exists a positive constant $K$ such that,
for every $n\ge1$,
%
%e2.5 #&#
\begin{equation}
\label{n_coercivity} 2xb_n(t,x) + (p_0-1)\bigl|
\sigma_n(t,x)\bigr|^2 \leq K\bigl(1+|x|^2\bigr)
\end{equation}
for any $t \in[0,T]$ and $x \in\mathbb R^d$.
\end{longlist}

%re2 #&#
\begin{remark} Note that the set of sequences of functions which
satisfy {B-1}--{B-3} is nonempty. In order to see this,
one considers the following.

\textit{Model} 1:
%
%e2.6 #&#
\begin{equation}
\label{b_n} b_n(t,x):=\frac{1}{1+n^{-\alpha}|b(t,x)| + n^{-\alpha}|\sigma(t,x)|^2}b(t,x)
\end{equation}
and
%
%e2.7 #&#
\begin{equation}
\label{sigma_n} \sigma_n(t,x):= \frac{1}{1+n^{-\alpha}|b(t,x)| + n^{-\alpha}|\sigma(t,x)|^2}\sigma(t,x),
\end{equation}
for any $t \in[0,T]$, $x \in\mathbb R^d$ and $n\ge1$. One observes
immediately that {B-2} is satisfied,
and furthermore that, due to {A-4}, {B-3} is also
satisfied. One also observes that, for every $R\ge0$,
\begin{eqnarray*}
&&\int_0^T\sup_{|x|\le R}\bigl|b_n(t,x)-
b(t,x)\bigr|^{p_0}\,dt
\nonumber
\\[-8pt]
\\[-8pt]
\nonumber
&&\qquad\le n^{-\alpha
p_0}\int_0^T
\sup_{|x|\le R}\frac{2^{p_0-1}(|b(t,x)|^{p_0} + |\sigma
(t,x)|^{2p_0})}{(1+n^{-\alpha}|b(t,x)| + n^{-\alpha}|\sigma
(t,x)|^2)^{p_0}}\bigl|b(t,x)\bigr|^{p_0} \,dt,
\end{eqnarray*}
which tends to 0 as $n\to\infty$, due to {A-2}. Similarly, one
obtains the same result for the diffusion coefficients so as to show
that {B-1} holds.

Finally, for every $n \geq1$, one deduces immediately that $b_n(t,x)$
and $\sigma_n(t,x)$ are
$\mathcal B(\mathbb R_+) \otimes\mathcal B(\mathbb R^d)$-measurable
functions which take values in $\mathbb R^d$ and $\mathbb R^{d
\times d_1}$, respectively.
\end{remark}

%re3 #&#
\begin{remark} \label{moment_N_bound}
Note that due to {B-2}, for each $n \ge1$, the norm of
$b_n$ and of $\sigma_n$ have at most linear growth in $x$ and that
guarantees the existence of a unique solution to \eqref{tem}.
Moreover, it guarantees along with {A-5} that for each $n\ge1$,
%
%e2.8 #&#
\begin{equation}
\label{n_bound} \sup_{0\le t\le T}\mathbb{E}\bigl[\bigl|X_n(t)\bigr|^p
\bigr] < \infty
\end{equation}
for any $p\le p_0$. Clearly, one cannot claim at this point that any of
these bounds is independent of $n$.
\end{remark}

The main results of this paper follow.

%th1 #&#
\begin{thmm}\label{mainthm}
Suppose \textup{{A-1}--{A-5}} and \textup{{B-1}--{B-3}} hold
with $\alpha\in(0,  1/2]$, then the numerical
scheme \eqref{tem} converges to the true solution of SDE \eqref{sde}
in $\mathcal L^p$-sense, that is,
\begin{eqnarray}
\lim_{n \rightarrow\infty}\sup_{0 \leq t \leq T}\mathbb E \bigl[
\bigl|X(t)-X_n(t)\bigr|^p \bigr]=0
\nonumber
\end{eqnarray}
for all $p <p_0$.
\end{thmm}

If one then moves from local to global monotonicity conditions and
considers coefficients which have at most polynomial growth, one
typically assumes the following:
%
%\begin{itemize}

A-6. There exist positive constants $l$ and $L$ such that,
for any $t \in[0,T]$,
\[
2(x-y) \bigl(b(t,x)-b(t,y)\bigr) + (p_1-1)\bigl|\sigma(t,x)-
\sigma(t,y)\bigr|^2 \leq L |x-y|^2
\]
and
\[
\bigl|b(t,x)-b(t,y)\bigr| \leq L\bigl(1+ |x|^l + |y|^l\bigr) |x-y|
\]
for all $x,  y \in\mathbb R^d$.

%re4 #&#
\begin{remark} \label{poly_growth_remark}
One observes that if {A-2}, {A-4} and {A-6} hold, then
%
%e2.9 #&#
\begin{eqnarray}
\label{poly_growth}
\bigl |b(t,x)\bigr| &\le&\bigl|b(t,x) - b(t, 0)\bigr| +
\bigl |b(t,0)\bigr| \le L\bigl(1 +
|x|^l\bigr)|x| + N_0
\nonumber
\\[-8pt]
\\[-8pt]
\nonumber
&\le& N\bigl(1 + |x|^{l+1}
\bigr)
\end{eqnarray}
for any $t\in[0, T]$ and $x\in\mathbb{R}^d$, where $N$ is a positive
constant. Similarly, one calculates
%
%e2.10 #&#
\begin{equation}
\label{diffusion_poly_growth} \bigl|\sigma(t,x)\bigr|^2 \le K \bigl(1+|x|^2
\bigr)+ 2N\bigl(1 + |x|^{l+1}\bigr)|x| \le C\bigl(1 + |x|^{l+2}
\bigr).
\end{equation}
\end{remark}

%re5 #&#
\begin{remark} Note that {A-6} and Remark~\ref
{poly_growth_remark} allow us to specify another model which produces
the optimal rate of convergence and satisfies {B-1}--{B-3}. Consider
the following.

\textit{Model} 2:
%
%e2.11 #&#
\begin{equation}
\label{b2_n} b_n(t,x):=\frac{1}{1+ n^{-\alpha}|x|^l}b(t,x)
\end{equation}
and
%
%e2.12 #&#
\begin{equation}
\label{sigma2_n} \sigma_n(t,x):= \frac{1}{1 + n^{-\alpha}|x|^l}\sigma(t,x),
\end{equation}
for any $t \in[0,T]$, $x \in\mathbb R^d$ and $n\ge1$. One then
observes that {B-2} is satisfied due to \eqref{poly_growth} and
\eqref{diffusion_poly_growth},
and furthermore that, due to {A-4}, {B-3} is also
satisfied. One also observes that, for every $R\ge0$,
\begin{eqnarray*}
\int_0^T \sup_{|x|\le R}\bigl|b_n(t,x)-
b(t,x)\bigr|^{p_0}\,dt &\le& n^{-\alpha
p_0}\int_0^T
\sup_{|x|\le R}\frac{|x|^{lp_0}}{(1 + n^{-\alpha
}|x|^l)^{p_0}}\bigl|b(t,x)\bigr|^{p_0} \,dt\\
& \to&0,
\end{eqnarray*}
as $n\to\infty$, due to \eqref{poly_growth}. Similarly, one obtains
the same result for the diffusion coefficients so as to show that
{B-1} holds.
\end{remark}

$\mathfrak{p}$-\textit{condition.} The coefficients $b_n$ and
$\sigma_n$ are given by equations \eqref{b2_n} and \eqref{sigma2_n}
with $\alpha=1/2$, $l\le\frac{p_0-2}{4}$ and there exists a positive
$p$ such that $p < p_1$ and $p\le\frac{p_0}{2l+1}$.

One then can recover the optimal rate of (strong) convergence for Euler
approximations.

%th2 #&#
\begin{thmm}\label{order-thmm}
Suppose \textup{{A-2}} and \textup{{A-4}--{A-6}} and the $\mathfrak
{p}$-\textit{condition} hold, then the
numerical scheme \eqref{tem} converges to the true solution of SDE
\eqref{sde} in $\mathcal L^p$-sense with order $1/2$, that is,
%
%e2.13 #&#
\begin{equation}
\label{rate_one_half} \sup_{0 \le t \le T}\mathbb E \bigl[
\bigl|X(t)-X_n(t)\bigr|^p \bigr] \le C n^{-p/2},
\end{equation}
where $C$ is a constant independent of $n$.
\end{thmm}

%re6 #&#
\begin{remark}
Observe that when $l=0$, that is, the drift and diffusion coefficients
are allowed to grow at most linearly and satisfy a global Lipschitz
condition, Theorem~\ref{order-thmm} produces the optimal result known in
classical literature, and thus it can be seen as a generalisation of
the classical approach since the restrictions in the
$\mathfrak{p}$-\textit{condition} are reduced to only one, namely $p\le p_0$.
\end{remark}

For somewhat smaller values of $p$, one obtains similar results in the
case of uniform $\mathcal L^p$ convergence.

%th3 #&#
\begin{thmm}\label{uniform-order-thmm}
Suppose \textup{{A-2}}, \textup{{A-4}--{A-6}} and the
$\mathfrak{p}$-\textit{condition} hold, then the
numerical scheme \eqref{tem} converges to the true solution of SDE
\eqref{sde} in \textit{uniform} $\mathcal L^q$-sense with order $1/2$,
that is,
%
%e2.14 #&#
\begin{equation}
\label{uniform_rate_one_half} \mathbb E \Bigl[\sup_{0 \le t \le T}
\bigl|X(t)-X_n(t)\bigr|^q \Bigr] \le C n^{-q/2},
\end{equation}
where $C$ is a constant independent of $n$, for all $q < p$.
\end{thmm}

%s3 #&#
\section{Convergence in probability and moment bounds}

One first notes the following result which along with the relevant
moment bounds of the numerical scheme \eqref{tem} suffice for the proof
of Theorem~\ref{mainthm}.

%th4 #&#
\begin{thmm} \label{cp_main}
Suppose conditions \textup{{A-1}--{A-4}} and \textup{{B-1}} hold.
Then the numerical scheme \eqref{tem} converges to the true solution of
SDE \eqref{sde} in probability, that is,
\[
\sup_{0 \le t\le T}\bigl|X_n(t)-X(t)\bigr| \mathop{\rightarrow}^{\mathbb{P}}
0 \qquad\mbox{as } n\to\infty.
\]
\end{thmm}
\begin{pf}
This is a direct consequence of Theorem~4.1 in \cite{Gyongy-Sabanis}.
\end{pf}

The $\mathcal L^2$ estimate is presented first as it demonstrates the
stability of the proposed numerical schemes. %and it is essential for
%the case where only the drift coefficient of SDE \eqref{sde} grows
%superlinearly (for %more details see Section~\ref{drift section}).

%le1 #&#
\begin{lemma}
\label{second_moment_bound} Consider the numerical scheme \eqref{tem}
and let \textup{{A-5}}, \textup{{B-2}} and \textup{{B-3}}
hold, then for some $C:=C(T, K,   \mathbb{E}[|X(0)|^2])$,
%
%e3.1 #&#
\begin{equation}
\label{2nd_moment_bound} \sup_{n\ge1}\sup_{0\le u\le T}
\mathbb{E}\bigl|X_n(u)\bigr|^2 <C.
\end{equation}
\end{lemma}
\begin{pf}
The application of It\^{o}'s formula yields
%
%e3.2 #&#
\begin{eqnarray}
\label{Ito_2} \bigl|X_n(t)\bigr|^2 &= &\bigl |X(0)\bigr|^2 + 2
\int_0^t X_n(s)
b_n\bigl(s, X_n\bigl(\kappa_n(s)\bigr)
\bigr) \,ds + \int_0^t\bigl| \sigma_n
\bigl(s, X_n\bigl(\kappa_n(s)\bigr)
\bigr)\bigr|^2 \,ds
\nonumber
\\
&&{} + 2 \int_0^t X_n(s)
\sigma_n\bigl(s, X_n\bigl(\kappa_n(s)
\bigr)\bigr) \,dW(s)
\nonumber
\\
&= & \bigl|X(0)\bigr|^2 + 2\int_0^t
\bigl[X_n\bigl(\kappa_n(s)\bigr)b_n
\bigl(s, X_n\bigl(\kappa_n(s)\bigr)\bigr)\\
&&{} +\bigl
\{X_n(s) - X_n\bigl(\kappa_n(s)\bigr)
\bigr\}b_n\bigl(s, X_n\bigl(\kappa _n(s)
\bigr)\bigr)\bigr] \,ds
\nonumber
\\
&&{} + \int_0^t \bigl| \sigma_n\bigl(s,
X_n\bigl(\kappa_n(s)\bigr)\bigr)\bigr|^2\,ds + 2
\int_0^t X_n(s)
\sigma_n\bigl(s, X_n\bigl(\kappa_n(s)
\bigr)\bigr) \,dW(s).\nonumber
\end{eqnarray}
Moreover, one calculates
%
%e3.3 #&#
\begin{eqnarray}
\label{local_second_moment}
&&\mathbb{E}\int_0^t \bigl
\{X_n(s) - X_n\bigl(\kappa_n(s)\bigr)
\bigr\}b_n\bigl(s, X_n\bigl(\kappa_n(s)
\bigr)\bigr) \,ds
\nonumber
\\
&&\qquad= \mathbb{E}\int_0^T \int
_{\kappa_n(s)}^{s}b_n\bigl(u,
X_n\bigl(\kappa_n(u)\bigr)\bigr)\,du b_n
\bigl(s, X_n\bigl(\kappa _n(s)\bigr)\bigr) \,ds
\nonumber
\\
&&\qquad\quad{}+ \mathbb{E}\int_0^t \int
_{\kappa_n(s)}^{s}\sigma_n\bigl(u,
X_n\bigl(\kappa _n(u)\bigr)\bigr)\,dW(u)
b_n\bigl(s, X_n\bigl(\kappa_n(u)\bigr)
\bigr) \,ds
\nonumber
\\
&&\qquad\le \mathbb{E}\int_0^T \int
_{\kappa_n(s)}^{s}\bigl|b_n\bigl(u,
X_n\bigl(\kappa_n(u)\bigr)\bigr)\bigr|\,du
\bigl|b_n\bigl(s, X_n\bigl(\kappa_n(s)\bigr)
\bigr)\bigr| \,ds
\nonumber
\\[-8pt]
\\[-8pt]
\nonumber
&&\qquad\quad{}+ \mathbb{E} \sum_{k=0}^{n(\lfloor t\rfloor+1)}\int
_{{k}/{n}}^{({(k+1)}/{n}) \wedge
t}\int_{{k}/{n}}^{s}
\sigma_n\bigl(u, X_n(k/n)\bigr)\,dW(u) b_n
\bigl(s, X_n(k/n)\bigr) \,ds
\nonumber
\\
&&\qquad\le Cn^{2\alpha}\mathbb{E}\int_0^t \int
_{\kappa_n(s)}^{s}\bigl(1+\bigl| X_n\bigl(
\kappa_n(u)\bigr)\bigr|\bigr)\,du \bigl(1+\bigl|X_n\bigl(
\kappa_n(s)\bigr)\bigr|\bigr) \,ds \nonumber\\
\eqntext{\mbox{(due to {B-2})}}
\\
&&\qquad\le C n^{2\alpha-1} \biggl(1 + \mathbb{E}\int_0^t
\bigl|X_n\bigl(\kappa_n(s)\bigr)\bigr|^2 \,ds
\biggr),\nonumber
\end{eqnarray}
where $C$ is a positive general constant independent of $n$.
Thus, due to \eqref{Ito_2}, {B-3}, \eqref{n_bound} and \eqref
{local_second_moment}, for any $t\in
[0, T]$,
\begin{eqnarray*}
\mathbb{E}\bigl|X_n(t)\bigr|^2 &\le& C \biggl(1 +
\mathbb{E}\bigl|X(0)\bigr|^2 + \mathbb{E}\int_0^t
\bigl|X_n\bigl(\kappa_n(s)\bigr)\bigr|^2 \,ds\biggr)
\\
&\le& C \biggl(1 + \mathbb{E}\bigl|X(0)\bigr|^2 + \int_0^t
\sup_{0\le u\le
s}\mathbb{E}\bigl|X_n(u)\bigr|^2 \,ds
\biggr),
\end{eqnarray*}
which implies
\[
\sup_{0\le u\le t}\mathbb{E}\bigl|X_n(u)\bigr|^2
\le C \biggl(1 + \mathbb{E}\bigl|X(0)\bigr|^2 + \int_0^t
\sup_{0\le u\le
s}\mathbb{E}\bigl|X_n(u)\bigr|^2 \,ds
\biggr) <\infty,
\]
where the positive general constant $C$ is independent of $n$. One then
observes that the application of Gronwall's lemma
yields the desired result.
\end{pf}

%le2 #&#
\begin{lemma} \label{optimal_moment_bound}
Suppose that \textup{{A-1}--{A-5}}, \textup{{B-2}} and \textup{{B-3}}
hold, then for every $p \le p_0$
%
%e3.4 #&#
\begin{equation}
\label{optimal_bound} \sup_{0\leq t \leq
T}\mathbb E \bigl|X(t)\bigr|^p \vee
\sup_{n \ge1}\sup_{0\leq t
\leq T}\mathbb E
\bigl|X_n(t)\bigr|^p <C,
\end{equation}
where the constant $C:=C(p, T, K,   \mathbb{E}[|X(0)|^p])$.
\end{lemma}
\begin{pf}
It is well known from the classical literature that the result
\[
\sup_{0\leq t \leq T}\mathbb E \bigl|X(t)\bigr|^p < C
\]
holds for every $p \le p_0$ when {A-1}--{A-5} hold. One
could consult, for example, \cite{Krylov-book} for more details or just
observe that the application of It\^{o}'s
formula to $|X(t)|^{p_0}$, along with {A-4}, {A-5} and the
application of Gronwall's and Fatou's lemmas yields the desired result.
Furthermore,
due to {B-2}, {B-3} and Remark~\ref{moment_N_bound}, one obtains
on the application of It\^{o}'s formula
%
%e3.5 #&#
\begin{eqnarray}
\label{Ito_p_0} &&\mathbb{E}\bigl|X_n(t)\bigr|^{p_0} \nonumber\\
&&\qquad\le
\mathbb{E}\bigl|X(0)\bigr|^{p_0} + \frac{p_0}{2}\mathbb{E}\int
_0^t\bigl |X_n(s)\bigr|^{p_0-2}K
\bigl(1+\bigl|X_n\bigl(\kappa_n(s)\bigr)\bigr|^2
\bigr) \,ds
\\
&&\qquad\quad{} + 2\mathbb {E}\int_0^t \bigl|X_n(s)\bigr|^{p_0-2}
\bigl\{X_n(s)-X_n\bigl(\kappa_n(s)
\bigr) \bigr\}b_n\bigl(s, X_n\bigl(
\kappa_n(s)\bigr)\bigr)\,ds.\nonumber
\end{eqnarray}
Thus, one needs to estimate the ``correction'' term
%
%e3.6 #&#
\begin{equation}
\label{E} E:= \mathbb{E}\int_0^t
\bigl|X_n(s)\bigr|^{p_0-2} \bigl\{X_n(s)-X_n
\bigl(\kappa _n(s)\bigr) \bigr\}b_n\bigl(s,
X_n\bigl(\kappa_n(s)\bigr)\bigr)\,ds.
\end{equation}
Then one calculates
%
%e3.7 #&#
\begin{eqnarray}
\label{E_1+E_2} E &= & \mathbb{E}\int_0^t
\bigl|X_n\bigl(\kappa_n(s)\bigr)\bigr|^{p_0-2} \bigl
\{X_n(s)-X_n\bigl(\kappa _n(s)\bigr)
\bigr\}b_n\bigl(s, X_n\bigl(\kappa_n(s)
\bigr)\bigr)\,ds
\nonumber
\\
&&{} + \mathbb{E}\int_0^t \bigl(\bigl|X_n(s)\bigr|^{p_0-2}
- \bigl|X_n\bigl(\kappa _n(s)\bigr)\bigr|^{p_0-2} \bigr)
\nonumber
\\[-8pt]
\\[-8pt]
\nonumber
&&{}\times\bigl\{X_n(s)-X_n\bigl(\kappa_n(s)
\bigr)\bigr\}b_n\bigl(s, X_n\bigl(
\kappa_n(s)\bigr)\bigr)\,ds
\\
&= & E_1 + E_2.\nonumber
\end{eqnarray}
Moreover, due to {B-2},
%
%e3.8 #&#
\begin{eqnarray}
\label{E_1} E_1&:= & \mathbb{E}\int_0^t
\bigl|X_n\bigl(\kappa_n(s)\bigr)\bigr|^{p_0-2} \bigl\{
X_n(s)-X_n\bigl(\kappa_n(s)\bigr)
\bigr\}b_n\bigl(s, X_n\bigl(\kappa_n(s)
\bigr)\bigr)\,ds
\nonumber
\\
&= & \mathbb{E}\int_0^t\bigl |X_n
\bigl(\kappa_n(s)\bigr)\bigr|^{p_0-2}\int_{\kappa
_n(s)}^{s}b_n
\bigl(u, X_n\bigl(\kappa_n(u)\bigr)\bigr)\,du
b_n\bigl(s, X_n\bigl(\kappa_n(s)\bigr)
\bigr)\,ds
\\
&&{} + \mathbb{E}\int_0^t \bigl|X_n
\bigl(\kappa_n(s)\bigr)\bigr|^{p_0-2}\nonumber\\
&&{}\times\int_{\kappa_n(s)}^{s}
\sigma_n\bigl(u, X_n\bigl(\kappa_n(u)
\bigr)\bigr)\,dW(u) b_n\bigl(s, X_n\bigl(\kappa
_n(s)\bigr)\bigr)\,ds
\nonumber
\\
&\le& \mathbb{E}\int_0^t \bigl|X_n
\bigl(\kappa _n(s)\bigr)\bigr|^{p_0-2}\nonumber\\
&&{}\times\int_{\kappa_n(s)}^{s}Cn^{\alpha}
\bigl(1 +\bigl|X_n\bigl(\kappa _n(u)\bigr)\bigr| \bigr)\,du
Cn^{\alpha} \bigl(1+ \bigl|X_n\bigl(\kappa_n(s)
\bigr)\bigr| \bigr)\,ds
\nonumber
\\
&\le& Cn^{2\alpha-1} \biggl(1 + \int_0^t
\mathbb{E}\bigl|X_n\bigl(\kappa_n(s)\bigr)\bigr|^{p_0}\,ds
\biggr)
\nonumber
\\
&\le& C \biggl(1 + \int_0^t \sup
_{r\le s}\mathbb{E}\bigl|X_n(r)\bigr|^{p_0}\,ds
\biggr).\nonumber
\end{eqnarray}
Furthermore, one uses It\^{o}'s formula for $p_0\ge 4$ in order to estimate $E_2$
[whereas for the case $2<p_0<4$, Lemma \ref{second_moment_bound} and 
the finiteness
of\break $\sup_{n\ge 1}\mathbb{E}\int_0^T |X_n(t) - X_n(\kappa_n(t))|^p |b_n(t,X_n(\kappa_n(t)))|^p\, dt$,
see Model 1 for example, are used to provide a uniform bound for \eqref{E}]. Note that
the case $p_0=2$ is covered by Lemma~\ref{second_moment_bound}:
\begin{eqnarray*}
E_2&:= & \mathbb{E}\int_0^t
\bigl(\bigl|X_n(s)\bigr|^{p_0-2} - \bigl|X_n\bigl(\kappa
_n(s)\bigr)\bigr|^{p_0-2} \bigr)\\
&&{}\times \bigl\{X_n(s)-X_n
\bigl(\kappa_n(s)\bigr) \bigr\}b_n\bigl(s,
X_n\bigl(\kappa_n(s)\bigr)\bigr)\,ds
\\
& =& \mathbb{E}\int_0^t \biggl[
(p_0-2)\int_{\kappa
_n(s)}^{s}\bigl|X_n(r)\bigr|^{p_0-4}X_n(r)b_n
\bigl(r, X_n\bigl(\kappa_n(r)\bigr)\bigr)\,dr
\\
& &{}+ (p_0-2) \biggl(\frac{p_0-2}{2}-1\biggr)\int
_{\kappa
_n(s)}^{s}\bigl|X_n(r)\bigr|^{p_0-6}\bigl|
\sigma^T_n\bigl(r, X_n\bigl(
\kappa_n(r)\bigr)\bigr)X_n(r)\bigr|^2\,dr\\
&&{} +\frac{(p_0-2)}{2}\int_{\kappa_n(s)}^{s}
\bigl|X_n(r)\bigr|^{p_0-4} \bigl|\sigma _n\bigl(r,
X_n\bigl(\kappa_n(r)\bigr)\bigr)\bigr|^2\,dr
\\
&&{} + (p_0-2)\int_{\kappa_n(s)}^{s}\bigl|X_n(r)\bigr|^{p_0-4}
X_n(r)\sigma_n\bigl(r, X_n\bigl(
\kappa_n(r)\bigr)\bigr) \,dW(r) \biggr]
\\
&&{} \times \biggl(\int_{\kappa_n(s)}^{s}b_n
\bigl(r, X_n\bigl(\kappa_n(r)\bigr)\bigr)\,dr + \int
_{\kappa_n(s)}^{s}\sigma_n\bigl(r,
X_n\bigl(\kappa_n(r)\bigr)\bigr)\,dW(r) \biggr)\\
&&{}\times
b_n\bigl(s, X_n\bigl(\kappa_n(s)\bigr)
\bigr)\,ds
\end{eqnarray*}
and thus
%
%e3.9 #&#
\begin{eqnarray}
\label{E_2} E_2 &\le& C \biggl(\mathbb{E}\int_0^t
\int_{\kappa
_n(s)}^{s}\bigl|X_n(r)\bigr|^{p_0-3}\bigl|b_n
\bigl(r, X_n\bigl(\kappa_n(r)\bigr)\bigr)\bigr|\,dr\nonumber\\
&&{}\times\int
_{\kappa
_n(s)}^{s}\bigl|b_n\bigl(r,
X_n\bigl(\kappa_n(r)\bigr)\bigr)\bigr|\,dr
\bigl|b_n\bigl(s, X_n\bigl(\kappa _n(s)
\bigr)\bigr)\bigr|\,ds
\\
&&{} + \mathbb{E}\int_0^t \int
_{\kappa_n(s)}^{s}\bigl|X_n(r)\bigr|^{p_0-3}\bigl|b_n
\bigl(r, X_n\bigl(\kappa_n(r)\bigr)\bigr)\bigr|\,dr\nonumber\\
&&{}\times \biggl|\int
_{\kappa_n(s)}^{s}\sigma_n\bigl(r,
X_n\bigl(\kappa _n(r)\bigr)\bigr)\,dW(r)\biggr|
\bigl|b_n\bigl(s, X_n\bigl(\kappa_n(s)\bigr)
\bigr)\bigr|\,ds
\nonumber
\\
&&{} + \mathbb{E}\int_0^t\int
_{\kappa_n(s)}^{s} \bigl|X_n(r)\bigr|^{p_0-4} \bigl|
\sigma _n\bigl(r, X_n\bigl(\kappa_n(r)
\bigr)\bigr)\bigr|^2\,dr\nonumber\\
&&{}\times \int_{\kappa_n(s)}^{s}\bigl|b_n
\bigl(r, X_n\bigl(\kappa _n(r)\bigr)
\bigr)\bigr|\,dr\bigl|b_n\bigl(s, X_n\bigl(\kappa_n(s)
\bigr)\bigr)\bigr|\,ds
\nonumber
\\
&&{} + \mathbb{E}\int_0^t\int
_{\kappa_n(s)}^{s} \bigl|X_n(r)\bigr|^{p_0-4} \bigl|
\sigma _n\bigl(r, X_n\bigl(\kappa_n(r)
\bigr)\bigr)\bigr|^2\,dr\nonumber
\\
&&{} \times\biggl|\int_{\kappa_n(s)}^{s}\sigma_n
\bigl(r, X_n\bigl(\kappa_n(r)\bigr)\bigr)\,dW(r)\biggr|
\bigl|b_n\bigl(s, X_n\bigl(\kappa_n(s)\bigr)
\bigr)\bigr|\,ds \biggr)
\nonumber
\\
&&{} + (p_0-2)\mathbb{E}\int_0^t
\int_{\kappa_n(s)}^{s}\bigl|X_n(r)\bigr|^{p_0-4}
X_n(r)\sigma_n\bigl(r, X_n\bigl(
\kappa_n(r)\bigr)\bigr) \,dW(r)
\nonumber
\\
&&{} \times\int_{\kappa_n(s)}^{s}b_n\bigl(r,
X_n\bigl(\kappa_n(r)\bigr)\bigr)\,dr\, b_n
\bigl(s, X_n\bigl(\kappa_n(s)\bigr)\bigr)\,ds
\nonumber
\\
& &{}+ (p_0-2)\mathbb{E}\int_0^t
\int_{\kappa_n(s)}^{s}\bigl|X_n(r)\bigr|^{p_0-4}
X_n(r)\sigma_n\bigl(r, X_n\bigl(
\kappa_n(r)\bigr)\bigr) \,dW(r)
\nonumber
\\
&&{} \times\int_{\kappa_n(s)}^{s}\sigma_n
\bigl(r, X_n\bigl(\kappa_n(r)\bigr)\bigr)\,dW(r)
b_n\bigl(s, X_n\bigl(\kappa_n(s)\bigr)
\bigr)\,ds
\nonumber
\\
& \le& C (E_{21} + E_{22} + E_{23} +
E_{24} ) + (p_0-2)E_{25} +
(p_0-2)E_{26}.\nonumber
\end{eqnarray}
One estimates $E_{21}$--$E_{26}$ by using Young's and H\"{o}lder's
inequalities as well as {B-2}. More precisely,
%
%e3.10 #&#
\begin{eqnarray}
\label{E_21} E_{21}&:= & \mathbb{E}\int_0^t
\int_{\kappa
_n(s)}^{s}\bigl|X_n(r)\bigr|^{p_0-3}\bigl|b_n
\bigl(r, X_n\bigl(\kappa_n(r)\bigr)\bigr)\bigr|\,dr
\nonumber
\\
& &{}\times\int_{\kappa_n(s)}^{s}\bigl|b_n\bigl(r,
X_n\bigl(\kappa_n(r)\bigr)\bigr)\bigr|\,dr
\bigl|b_n\bigl(s, X_n\bigl(\kappa_n(s)\bigr)
\bigr)\bigr|\,ds
\nonumber
\\
& \le&\mathbb{E}\int_0^t Cn^{3\alpha-1}\int
_{\kappa
_n(s)}^{s}\bigl|X_n(r)\bigr|^{p_0-3}
\bigl(1 +\bigl|X_n\bigl(\kappa_n(s)\bigr)\bigr|
\bigr)^3\,dr\,ds
\nonumber
\\[-8pt]
\\[-8pt]
\nonumber
& \le& Cn^{3\alpha-2} \biggl(1 + \int_0^t
\sup_{r\le s}\mathbb {E}\bigl|X_n(r)\bigr|^{p_0}\,ds \\
&&{}+
\int_0^t \mathbb{E}\bigl|X_n\bigl(
\kappa_n(s)\bigr)\bigr|^{p_0}\,ds \biggr)
\nonumber
\\
& \le& C \biggl(1 + \int_0^t \sup
_{r\le s}\mathbb {E}\bigl|X_n(r)\bigr|^{p_0}\,ds
\biggr),\nonumber
\end{eqnarray}
and
\begin{eqnarray*}
E_{22}&:= & \mathbb{E}\int_0^t \int
_{\kappa
_n(s)}^{s}\bigl|X_n(r)\bigr|^{p_0-3}\bigl|b_n
\bigl(r, X_n\bigl(\kappa_n(r)\bigr)\bigr)\bigr|\,dr
\\
& &{}\times\biggl|\int_{\kappa_n(s)}^{s}\sigma_n
\bigl(r, X_n\bigl(\kappa_n(r)\bigr)\bigr)\,dW(r)\biggr|
\bigl|b_n\bigl(s, X_n\bigl(\kappa_n(s)\bigr)
\bigr)\bigr|\,ds
\\
\hspace*{-1pt}& \le&\hspace*{-1pt}\mathbb{E}\int_0^t\hspace*{-0.5pt} \biggl\{ \biggl(\int
_{\kappa
_n(s)}^{s}\hspace*{-0.5pt}\bigl|X_n(r)\bigr|^{p_0-3}\bigl|b_n
\bigl(r, X_n\bigl(\kappa_n(r)\bigr)
\bigr)\bigr|\,dr\bigl|b_n\bigl(s, X_n\bigl(\kappa
_n(s)\bigr)\bigr)\bigr|\hspace*{-1pt} \biggr)^{{p_0}/{(p_0-1)}}
\\
&&{} + \biggl|\int_{\kappa_n(s)}^{s}\sigma_n\bigl(r,
X_n\bigl(\kappa_n(r)\bigr)\bigr)\,dW(r)
\biggr|^{p_0} \biggr\}\,ds
\\
& \le&\mathbb{E}\int_0^t \biggl\{
\biggl(Cn^{2\alpha}\int_{\kappa
_n(s)}^{s}\bigl|X_n(r)\bigr|^{p_0-3}
\bigl(1 +\bigl|X_n\bigl(\kappa_n(s)\bigr)\bigr|
\bigr)^2\,dr \biggr)^{{p_0}/{(p_0-1)}}
\\
&&{} + \biggl|\int_{\kappa_n(s)}^{s}\sigma_n\bigl(r,
X_n\bigl(\kappa_n(r)\bigr)\bigr)\,dW(r)
\biggr|^{p_0} \biggr\} \,ds
\\
& \le&\mathbb{E}\int_0^t \biggl(Cn^{2\alpha}
\int_{\kappa_n(s)}^{s}\bigl(1 + \bigl|X_n(r)\bigr|^{p_0-1}
+ \bigl|X_n\bigl(\kappa_n(s)\bigr)\bigr|^{p_0-1}\bigr)\,dr
\biggr)^{{p_0}/{(p_0-1)}} \,ds
\\
&&{} + \int_0^t\mathbb{E} \biggl(\int
_{\kappa_n(s)}^{s}\bigl|\sigma_n\bigl(r,
X_n\bigl(\kappa _n(r)\bigr)\bigr)\bigr|^2\,dr
\biggr)^{p_0/2}\,ds
\\
& \le& Cn^{(2\alpha-1)({p_0}/{(p_0-1)})}\int_0^t \Bigl(1 +
\sup_{r\le
s}\mathbb{E}\bigl|X_n(r)\bigr|^{p_0} +
\mathbb{E}\bigl|X_n\bigl(\kappa_n(s)\bigr)\bigr|^{p_0}
\Bigr)\,ds
\\
&&{} + \int_0^t \mathbb{E} \biggl(\int
_{\kappa_n(s)}^{s}Cn^{\alpha}\bigl(1 +
\bigl|X_n\bigl(\kappa_n(r)\bigr)\bigr|^2\bigr)\,dr
\biggr)^{p_0/2}\,ds
\\
& \le& C \biggl(1 + \int_0^t \sup
_{r\le s}\mathbb{E}\bigl|X_n(r)\bigr|^{p_0} \,ds
\biggr)\\
&&{} + Cn^{(\alpha-1)({p_0}/{2})} \biggl(1 + \int_0^t
\mathbb{E} \bigl|X_n\bigl(\kappa _n(s)\bigr)\bigr|^{p_0}\,ds
\biggr),
\end{eqnarray*}
which yields
%
%e3.11 #&#
\begin{equation}
\label{E_22} E_{22} \le C \biggl(1 + \int_0^t
\sup_{r\le s}\mathbb{E}\bigl|X_n(r)\bigr|^{p_0} \,dr
\biggr).
\end{equation}
Furthermore,
%
%e3.12 #&#
\begin{eqnarray}
\label{E_23} E_{23}&:= & \mathbb{E}\int_0^t
\int_{\kappa_n(s)}^{s} \bigl|X_n(r)\bigr|^{p_0-4}
\bigl|\sigma_n\bigl(r, X_n\bigl(\kappa_n(r)
\bigr)\bigr)\bigr|^2\,dr
\nonumber
\\
&&{} \times\int_{\kappa_n(s)}^{s}\bigl|b_n\bigl(r,
X_n\bigl(\kappa_n(r)\bigr)\bigr)\bigr|\,dr\bigl|b_n
\bigl(s, X_n\bigl(\kappa_n(s)\bigr)\bigr)\bigr|\,ds
\\
& \le&\mathbb{E}\int_0^t Cn^{3\alpha-1}\int
_{\kappa
_n(s)}^{s}\bigl|X_n(r)\bigr|^{p_0-4}
\bigl(1 +\bigl|X_n\bigl(\kappa_n(s)\bigr)\bigr|^2
\bigr)
\nonumber
\\
&&{}\times \bigl(1 +\bigl|X_n\bigl(\kappa_n(s)\bigr)\bigr|
\bigr)^2\,dr\,ds\nonumber
\\
& \le& Cn^{3\alpha-2} \biggl(1 + \int_0^t
\sup_{r\le s}\mathbb{E}\bigl|X_n(r)\bigr|^{p_0}\,ds +
\int_0^t \mathbb {E}\bigl|X_n\bigl(
\kappa_n(s)\bigr)\bigr|^{p_0}\,ds \biggr)
\nonumber
\\
& \le& C \biggl(1 + \int_0^t \sup
_{r\le s}\mathbb{E}\bigl|X_n(r)\bigr|^{p_0}\,ds \biggr)\nonumber
\end{eqnarray}
and
\begin{eqnarray*}
E_{24}&:= & \mathbb{E}\int_0^t\int
_{\kappa_n(s)}^{s} \bigl|X_n(r)\bigr|^{p_0-4} \bigl|
\sigma_n\bigl(r, X_n\bigl(\kappa_n(r)
\bigr)\bigr)\bigr|^2\,dr
\\
&&{} \times\biggl|\int_{\kappa_n(s)}^{s}\sigma_n
\bigl(r, X_n\bigl(\kappa_n(r)\bigr)\bigr)\,dW(r)\biggr|
\bigl|b_n\bigl(s, X_n\bigl(\kappa_n(s)\bigr)
\bigr)\bigr|\,ds
\\
& \le&\mathbb{E}\int_0^t \biggl\{ \biggl(\int
_{\kappa_n(s)}^{s} \bigl|X_n(r)\bigr|^{p_0-4}\bigl |
\sigma_n\bigl(r, X_n\bigl(\kappa_n(r)
\bigr)\bigr)\bigr|^2\,dr\\
&&{}\times\bigl|b_n\bigl(s, X_n\bigl(
\kappa _n(s)\bigr)\bigr)\bigr| \biggr)^{{p_0}/{(p_0-1)}}
\\
&&{} + \biggl|\int_{\kappa_n(s)}^{s}\sigma_n\bigl(r,
X_n\bigl(\kappa_n(r)\bigr)\bigr)\,dW(r)
\biggr|^{p_0} \biggr\}\,ds
\\
& \le&\int_0^t\mathbb{E} \biggl[\int
_{\kappa_n(s)}^{s} \bigl|X_n(r)\bigr|^{p_0-4}
Cn^{\alpha} \bigl(1 + \bigl|X_n\bigl(\kappa_n(r)
\bigr)\bigr|^2 \bigr)\,dr\\
&&{}\times Cn^{\alpha} \bigl(1 + \bigl|X_n
\bigl(\kappa_n(s)\bigr)\bigr| \bigr) \biggr]^{{p_0}/{(p_0-1)}}\,ds
\\
&&{} +\int_0^t \mathbb{E} \biggl|\int
_{\kappa_n(s)}^{s}\sigma_n\bigl(r,
X_n\bigl(\kappa _n(r)\bigr)\bigr)\,dW(r)
\biggr|^{p_0}\,ds
\\
& \le&\mathbb{E}\int_0^t \biggl(Cn^{2\alpha}
\int_{\kappa_n(s)}^{s}\bigl(1 + \bigl|X_n(r)\bigr|^{p_0-1}
+ \bigl|X_n\bigl(\kappa_n(s)\bigr)\bigr|^{p_0-1}\bigr)\,dr
\biggr)^{{p_0}/{(p_0-1)}} \,ds
\\
&&{} + \int_0^t\mathbb{E} \biggl(\int
_{\kappa_n(s)}^{s}\bigl|\sigma_n\bigl(r,
X_n\bigl(\kappa _n(r)\bigr)\bigr)\bigr|^2\,dr
\biggr)^{p_0/2}\,ds
\\
& \le& Cn^{(2\alpha-1)({p_0}/{(p_0-1)})}\int_0^t \Bigl(1 +
\sup_{r\le
s}\mathbb{E}\bigl|X_n(r)\bigr|^{p_0} +
\mathbb{E}\bigl|X_n\bigl(\kappa_n(s)\bigr)\bigr|^{p_0}
\Bigr)\,ds
\\
&&{} + \int_0^t \mathbb{E} \biggl(\int
_{\kappa_n(s)}^{s}Cn^{\alpha}\bigl(1 +
\bigl|X_n\bigl(\kappa_n(r)\bigr)\bigr|^2\bigr)\,dr
\biggr)^{p_0/2}\,ds
\\
& \le& C \biggl(1 + \int_0^t \sup
_{r\le s}\mathbb{E}\bigl|X_n(r)\bigr|^{p_0} \,ds
\biggr)\\
&&{} + Cn^{(\alpha-1)({p_0}/{2})} \biggl(1 + \int_0^t
\mathbb{E} \bigl|X_n\bigl(\kappa _n(s)\bigr)\bigr|^{p_0}\,ds
\biggr),
\end{eqnarray*}
which also yields
%
%e3.13 #&#
\begin{equation}
\label{E_24} E_{24} \le C \biggl(1 + \int_0^t
\sup_{r\le s}\mathbb{E}\bigl|X_n(r)\bigr|^{p_0} \,dr
\biggr).
\end{equation}
Finally,
%
%e3.14 #&#
\begin{eqnarray}
\label{E_25} E_{25}&:= &\mathbb{E}\int_0^t
\int_{\kappa_n(s)}^{s}\bigl|X_n(r)\bigr|^{p_0-4}
X_n(r)\sigma_n\bigl(r, X_n\bigl(
\kappa_n(r)\bigr)\bigr) \,dW(r)
\nonumber
\\[-8pt]
\\[-8pt]
\nonumber
&&{} \times\int_{\kappa_n(s)}^{s}b_n\bigl(r,
X_n\bigl(\kappa_n(r)\bigr)\bigr)\,dr\, b_n
\bigl(s, X_n\bigl(\kappa_n(s)\bigr)\bigr)\,ds = 0
\end{eqnarray}
and
%
%e3.15 #&#
\begin{eqnarray}
\label{E_26} E_{26}&:= & \mathbb{E}\int_0^t
\int_{\kappa_n(s)}^{s}\bigl|X_n(r)\bigr|^{p_0-4}
X_n(r)\sigma_n\bigl(r, X_n\bigl(
\kappa_n(r)\bigr)\bigr) \,dW(r)
\nonumber
\\
& &{}\times\int_{\kappa_n(s)}^{s}\sigma_n
\bigl(r, X_n\bigl(\kappa_n(r)\bigr)\bigr)\,dW(r)
b_n\bigl(s, X_n\bigl(\kappa_n(s)\bigr)
\bigr)\,ds
\nonumber
\\
&= & \mathbb{E}\int_0^t \int
_{\kappa_n(s)}^{s}\bigl|X_n(r)\bigr|^{p_0-4}
X_n(r)\sigma_n\bigl(r, X_n\bigl(
\kappa_n(r)\bigr)\bigr)\sigma_n^T\bigl(r,
X_n\bigl(\kappa_n(r)\bigr)\bigr)\,dr\nonumber\\
&&{}\times  b_n
\bigl(s, X_n\bigl(\kappa _n(s)\bigr)\bigr)\,ds
\nonumber
\\[-8pt]
\\[-8pt]
\nonumber
& \le& \mathbb{E}\int_0^t \int
_{\kappa_n(s)}^{s}\bigl|X_n(r)\bigr|^{p_0-3}\bigl|
\sigma_n\bigl(r, X_n\bigl(\kappa_n(r)
\bigr)\bigr)\bigr|^2\,dr \bigl|b_n\bigl(s, X_n\bigl(
\kappa_n(s)\bigr)\bigr)\bigr|\,ds
\\
& \le& \mathbb{E}\int_0^t Cn^{2\alpha}
\int_{\kappa
_n(s)}^{s}\bigl|X_n(r)\bigr|^{p_0-3}
\bigl(1 +\bigl|X_n\bigl(\kappa_n(s)\bigr)\bigr|^2
\bigr) \bigl(1 +\bigl|X_n\bigl(\kappa_n(s)\bigr)\bigr|
\bigr)\,dr\,ds
\nonumber
\\
& \le& Cn^{2\alpha-1} \biggl(1 + \int_0^t
\sup_{r\le s}\mathbb{E}\bigl|X_n(r)\bigr|^{p_0}\,ds +
\int_0^t \mathbb {E}\bigl|X_n\bigl(
\kappa_n(s)\bigr)\bigr|^{p_0}\,ds \biggr)
\nonumber
\\
& \le& C \biggl(1 + \int_0^t \sup
_{r\le s}\mathbb{E}\bigl|X_n(r)\bigr|^{p_0}\,ds
\biggr).\nonumber
\end{eqnarray}
Thus, due to \eqref{E_21}--\eqref{E_26}, \eqref{E_1}, \eqref{E_2} and
\eqref{E_1+E_2},
\begin{eqnarray*}
&&\mathbb{E}\int_0^t \bigl|X_n(s)\bigr|^{p_0-2}
\bigl\{X_n(s)-X_n\bigl(\kappa_n(s)
\bigr) \bigr\} b_n\bigl(s, X_n\bigl(
\kappa_n(s)\bigr)\bigr)\,ds\\
&&\qquad \le C \biggl(1 + \int_0^t
\sup_{r\le s}\mathbb {E}\bigl|X_n(r)\bigr|^{p_0}\,ds
\biggr),
\end{eqnarray*}
which yields due to \eqref{Ito_p_0} and Young's inequality that
%
%e3.16 #&#
\begin{eqnarray}
\label{p0-th moment}
\mathbb{E}\bigl|X_n(t)\bigr|^{p_0}& \le& C \biggl(1 +
\mathbb{E}\bigl|X(0)\bigr|^{p_0} + \mathbb{E}\int_0^t
\bigl|X_n(s)\bigr|^{p_0} \,ds \nonumber\\
&&{}+ \mathbb{E}\int_0^t
\bigl(1+\bigl|X_n\bigl(\kappa_n(s)\bigr)\bigr|^2
\bigr)^{p_0/2} \,ds\biggr)
\nonumber
\\[-8pt]
\\[-8pt]
\nonumber
&&{} + 2\mathbb{E}\int_0^t \bigl|X_n(s)\bigr|^{p_0-2}
\bigl\{X_n(s)-X_n\bigl(\kappa_n(s)
\bigr) \bigr\}b_n\bigl(s, X_n\bigl(
\kappa_n(s)\bigr)\bigr)\,ds
\\
&\le& C \biggl(1 + \mathbb{E}\bigl|X(0)\bigr|^{p_0} + \int_0^t
\sup_{0\le u\le
s}\mathbb{E}\bigl|X_n(u)\bigr|^{p_0} \,ds
\biggr)<\infty\nonumber
\end{eqnarray}
due to \eqref{n_bound}. The application of Gronwall's lemma yields the
desired result.
\end{pf}

%re7 #&#
\begin{remark} \label{split}
In order to ease notation, it is chosen not to explicitly present the
calculations for, and thus it is left as an exercise to the reader, the
case where the drift and the diffusion coefficient(s) have the
following representation:
\[
b(t,x) = b^{(1)}(t,x)+b^{(2)}(t,x)\quad \mbox{and}\quad \sigma(t,x) =
\sigma ^{(1)}(t,x)+\sigma^{(2)}(t,x),
\]
where $b^{(1)}(t,x)$ and $\sigma^{(1)}(t,x)$ are Lipschitz continuous
and grow at most linearly (in $x$) and the nonlinearities, that is,
super-linear growth, appear in $b^{(2)}(t,x)$ and in $\sigma
^{(2)}(t,x)$. In such a case, the analysis for $b^{(1)}(t,x)$ and
$\sigma^{(1)}(t,x)$ follows closely the classical approach, see also
``correction'' term $E$ in \eqref{E}. Note also that in such a case, $b(t,x)$
and $\sigma(t,x)$ are replaced by $b^{(2)}(t,x)$ and $\sigma
^{(2)}(t,x)$, respectively, in \eqref{b_n}, \eqref{sigma_n}, \eqref
{b2_n} and \eqref{sigma2_n}.
\end{remark}

%s4 #&#
\section{Proof of main results}

%s4.1 #&#
\subsection{$\mathcal{L}^p$-convergence}
\mbox{}
\begin{pf*}{Proof of Theorem~\ref{mainthm}}
This is now a direct consequence of Theorem~\ref{cp_main} and Lemma~\ref
{optimal_moment_bound}.
\end{pf*}

%le3 #&#
\begin{lemma}
\label{b_n_to_b}
Consider the numerical scheme \eqref{tem} with coefficients $b_n$ and
$\sigma_n$ given by \eqref{b2_n} and \eqref{sigma2_n}, respectively.
Suppose \textup{{A-2}}, \textup{{A-4}--{A-6}} and $p\le\frac
{p_0}{2l+1}$. Then
%
%e4.1 #&#
\begin{equation}
\label{difference_b} \mathbb{E}\biggl[\int_0^T\bigl |b
\bigl(s, X_n\bigl(\kappa_n(s)\bigr)\bigr)-
b_n\bigl(s, X_n\bigl(\kappa _n(s)\bigr)
\bigr)\bigr|^p\,ds \biggr] \le C n^{-\alpha p}
\end{equation}
and
%
%e4.2 #&#
\begin{equation}
\label{difference_sigma} \mathbb{E}\biggl[\int_0^T\bigl |
\sigma\bigl(s, X_n\bigl(\kappa_n(s )\bigr)\bigr) -
\sigma_n\bigl(s, X_n\bigl(\kappa_n(s )
\bigr)\bigr)\bigr|^p\,ds\biggr] \le C n^{-\alpha p},
\end{equation}
where $C$ is a constant independent of $n$.
\end{lemma}
\begin{pf}
One immediately observes that, due to \eqref{poly_growth}, \eqref
{diffusion_poly_growth}, \eqref{b2_n} and \eqref{sigma2_n}
\begin{eqnarray*}
&&\mathbb{E}\biggl[\int_0^T \bigl|b\bigl(s,
X_n\bigl(\kappa_n(s)\bigr)\bigr)-b_n
\bigl(s, X_n\bigl(\kappa _n(s)\bigr)
\bigr)\bigr|^p \,ds\biggr]
\\
&&\qquad\le n^{-\alpha p}\mathbb{E} \biggl[\int_0^T
\frac{|X_n(\kappa
_n(s))|^{lp}}{(1 + n^{-\alpha}\bigl|X_n(\kappa
_n(s))\bigr|^{l})^{p}}\bigl|b\bigl(t,X_n\bigl(\kappa_n(s)\bigr)
\bigr)\bigr|^{p} \,dt\biggr]
\\
&&\qquad\le Cn^{-\alpha p}\mathbb{E} \biggl[\int_0^{T}
\bigl|X_n\bigl(\kappa_n(s)\bigr)\bigr|^{lp}
\bigl(1+\bigl|X_n\bigl(\kappa_n(s)\bigr)\bigr|^{l+1}
\bigr)^p\,ds \biggr],
\end{eqnarray*}
which implies \eqref{difference_b} due to Lemma~\ref
{optimal_moment_bound} and the assumption that $p\le\frac{p_0}{2l+1}$.
One applies the same technique in order to obtain (\ref{difference_sigma}).
\end{pf}

%le4 #&#
\begin{lemma} \label{one-step-rate_lemma}
Consider the numerical scheme \eqref{tem}. Let \textup{{A-2}},
\textup{{A-4}--{A-6}} and \textup{{B-2}} with $\alpha=1/2$ hold, then for
any positive $p \le\max(2, \frac{2p_0}{l+2})$ and $l\le p_0-2$,
%
%e4.3 #&#
\begin{equation}
\label{one-step-rate} \sup_{0\le t \le T}\mathbb E \bigl|X_n(t)-X_n
\bigl(\kappa_n(t)\bigr)\bigr|^p \le Cn^{-p/2},
\end{equation}
where $C$ is a positive constant independent of $n$.
\end{lemma}

\begin{pf}
For any $p \in[1, \frac{2p_0}{l+2}]$ and every $t\in[0, T]$,
\begin{eqnarray*}
&&\mathbb E \bigl|X_n(t)-X_n\bigl(\kappa_n(t)
\bigr)\bigr|^{p}
\\
&&\qquad= \mathbb E\biggl| \int_{\kappa_n(t)}^{t}b_n
\bigl(r, X_n\bigl(\kappa_n(r)\bigr)\bigr)\,dr+\int
_{\kappa_n(t)}^{t}\sigma_n\bigl(r,
X_n\bigl(\kappa_n(r)\bigr)\bigr)\,dW(r) \biggr|^p
\end{eqnarray*}
and thus, due to H\"older's inequality,
%
%e4.4 #&#
\begin{eqnarray}
\label{g1} &&\mathbb E \bigl|X_n(t)-X_n\bigl(
\kappa_n(t)\bigr)\bigr|^p\nonumber\\
&&\qquad \le 2^{p-1} \bigl|t-
\kappa_n(t)\bigr|^{p-1}\mathbb{E}\int_{\kappa_n(t)}^t
\bigl|b_n\bigl(r, X_n\bigl(\kappa _n(r)
\bigr)\bigr)\bigr|^p \,dr
\\
&&\qquad\quad{} + 2^{p-1} \mathbb E \biggl|\int_{\kappa_n(t)}^t
\sigma_n\bigl(r, X_n\bigl(\kappa _n(r)
\bigr)\bigr)\,dW(r) \biggr|^p.\nonumber
\end{eqnarray}
One then observes that, due to {B-2},
%
%e4.5 #&#
\begin{eqnarray}
\label{g2}&& 2^{p-1} \bigl|t-\kappa_n(t)\bigr|^{p-1}
\mathbb E\int_{\kappa_n(t)}^t \bigl|b_n\bigl(r,
X_n\bigl(\kappa_n(r)\bigr)\bigr)\bigr|^p \,dr \nonumber\\
&&\qquad\le \biggl(\frac{2}{n} \biggr)^{p-1} \mathbb E\int
_{\kappa_n(t)}^t n^{\alpha
p} \bigl(1+
\bigl|X_n\bigl(\kappa_n(r)\bigr)\bigr| \bigr)^p \,dr
\\
&&\qquad \le C n^{(\alpha- 1)p}\nonumber
\end{eqnarray}
and, due to \eqref{diffusion_poly_growth}, one obtains
%
%e4.6 #&#
\begin{eqnarray}
\label{g3a}&& \mathbb E \biggl|\int_{\kappa_n(t)}^t
\sigma_n\bigl(r, X_n\bigl(\kappa_n(r)
\bigr)\bigr)\,dW(r)\biggr |^p \nonumber\\
&&\qquad\le C \mathbb E \biggl[ \biggl(\int
_{\kappa_n(t)}^t\bigl|\sigma_n\bigl(r,
X_n\bigl(\kappa_n(r)\bigr)\bigr)\bigr|^2 \,dr
\biggr)^{p/2} \biggr]
\\
&&\qquad\le C E \biggl[ \biggl(\int_{\kappa_n(t)}^t \bigl(1+
\bigl|X_n\bigl(\kappa _n(r)\bigr)\bigr|^{l+2} \bigr)
\,dr \biggr)^{p/2} \biggr] \le Cn^{-p/2}.\nonumber
\end{eqnarray}
This due to the fact that for the case $p>2$, H\"{o}lder's inequality
gives the desired result as $p \le\frac{2p_0}{l+2}$ and thus $\frac
{l+2}{2}p\le p_0$, and for the case $1\le p \le2$, one uses Jensen's
inequality for concave functions and/or the fact that $l\le p_0-2$.
Substituting \eqref{g2} and \eqref{g3a} in \eqref{g1} yields \eqref
{one-step-rate}. Similarly, one obtains the same result for $0<p<1$,
due to Jensen's inequality for concave functions, $l\le p_0-2$ and
\[
\mathbb E\bigl |X_n(t)-X_n\bigl(\kappa_n(t)
\bigr)\bigr|^{p} \le \bigl(\mathbb E\bigl |X_n(t)-X_n
\bigl(\kappa_n(t)\bigr)\bigr| \bigr)^p \le
\bigl(Cn^{-1/2}\bigr)^p.
\]
\upqed\end{pf}

\begin{pf*}{Proof of Theorem~\ref{order-thmm}}
One considers first, for every $n\ge1$ and $t\in[0,  T]$,
%
%e4.7 #&#
\begin{equation}
\label{def_chi_beta} \chi_n(t):=X(t)-X_n(t),\qquad
\beta_n(t):= b\bigl(t, X(t)\bigr)-b_n
\bigl(t,X_n\bigl(\kappa_n(t)\bigr)\bigr)
\end{equation}
and
%
%e4.8 #&#
\begin{equation}
\label{def_alpha} \alpha_n(t):=\sigma\bigl(t,X(t)\bigr)-
\sigma_n\bigl(t,X_n\bigl(\kappa_n(t)
\bigr)\bigr)
\end{equation}
to obtain for any $p\ge2$
%
%e4.9 #&#
\begin{eqnarray}
\label{xi} \bigl|\chi_n(t)\bigr|^p& \le&\frac{p}{2}
\int_{0}^t\bigl|\chi_n(s)\bigr|^{p-2}
\bigl[ 2\chi _n(s)\beta_n(s) + (p-1)\bigl|
\alpha_n(s)\bigr|^2 \bigr] \,ds
\nonumber
\\[-8pt]
\\[-8pt]
\nonumber
&&{}+ p \int_{0}^t\bigl|
\chi_n(s)\bigr|^{p-2}\chi _n(s)
\alpha_n(s) \,dW(s).
\end{eqnarray}
One then observes, for any $\varepsilon>0$,
%
%e4.10 #&#
\begin{eqnarray}
\label{sum-estimate}&& 2\chi_n(s)\beta_n(s)+ (p-1)\bigl|
\alpha_n(s)\bigr|^2\nonumber \\
&&\qquad=  2\bigl[X(s)-X_n(s)
\bigr] \bigl[b\bigl(s, X(s)\bigr)-b\bigl(s,X_n(s)\bigr)\bigr]
\nonumber
\\
&&\qquad\quad{} + 2\bigl[X(s)-X_n(s)\bigr] \bigl[b\bigl(s, X_n(s)
\bigr)-b\bigl(s,X_n\bigl(\kappa_n(s)\bigr)\bigr)\bigr]
\nonumber
\\
&&\qquad\quad{} + 2\bigl[X(s) - X_n(s)\bigr] \bigl[b\bigl(s,X_n
\bigl(\kappa_n(s)\bigr)\bigr)- b_n
\bigl(s,X_n\bigl(\kappa_n(s)\bigr)\bigr)\bigr]
\\
&&\qquad\quad{} + (1+\varepsilon) (p-1)\bigl|\sigma\bigl(s,X(s)\bigr)-\sigma\bigl(s,X_n(s)
\bigr)\bigr|^2
\nonumber
\\
&&\qquad\quad{} + 2\biggl(1+\frac{1}{\varepsilon}\biggr) (p-1)\bigl|\sigma\bigl(s,X_n(s)
\bigr)-\sigma \bigl(s,X_n\bigl(\kappa_n(s)\bigr)
\bigr)\bigr|^2
\nonumber
\\
&&\qquad\quad{} + 2\biggl(1+\frac{1}{\varepsilon}\biggr) (p-1)\bigl|\sigma\bigl(s,X_n
\bigl(\kappa _n(s)\bigr)\bigr)-\sigma_n
\bigl(s,X_n\bigl(\kappa_n(s)\bigr)
\bigr)\bigr|^2.\nonumber
\end{eqnarray}
One further observes that
\begin{eqnarray*}
&&(p_1-1)\bigl|\sigma(t,x)-\sigma(t,y)\bigr|^2\\
&&\qquad \le
L|x-y|^2-2(x-y) \bigl(b(t,x)-b(t,y)\bigr) \qquad\mbox{(due to
{A-6})}
\\
&&\qquad \le C\bigl(1+|x|^{l}+|y|^{l}\bigr)|x-y|^2
\end{eqnarray*}
and thus, due to {A-2}, {A-4}, {A-6} and the fact
that there exists an $\varepsilon$ such that $(1+\varepsilon)(p-1)\le p_1-1$
since it is assumed that $p < p_1$, estimate \eqref{sum-estimate} yields
%
%e4.11 #&#
\begin{eqnarray}
\label{sum-difference}&& 2\chi_n(s)\beta_n(s)+ (p-1)\bigl|
\alpha_n(s)\bigr|^2 \nonumber\\
&&\qquad\le C\bigl| \chi_n(s)\bigr|^2
+ C\bigl(1+ \bigl|X_n(s)\bigr|^{2l} +\bigl|X_n\bigl(
\kappa_n(s)\bigr)\bigr|^{2l}\bigr)
\nonumber
\\[-8pt]
\\[-8pt]
\nonumber
& &\qquad\quad{}\times\bigl|X_n(s)-X_n\bigl(\kappa_n(s)
\bigr)\bigr|^2+\bigl|b\bigl(s, X_n\bigl(\kappa _n(s)
\bigr)\bigr)-b_n\bigl(s,X_n\bigl(\kappa_n(s)
\bigr)\bigr)\bigr|^2
\nonumber
\\
& &\qquad\quad{}+ C\bigl|\sigma\bigl(s,X\bigl(\kappa_n(s)\bigr)\bigr)-
\sigma_n\bigl(s,X_n\bigl(\kappa_n(s)
\bigr)\bigr)\bigr|^2.\nonumber
\end{eqnarray}
Furthermore, by taking into consideration \eqref{xi}, \eqref{sum-difference},
Remark~\ref{moment_N_bound} and \eqref{difference_sigma}, one
obtains that
\begin{eqnarray*}
&&\mathbb{E}\bigl|\chi_n(t)\bigr|^p \\
&&\qquad\le  C\mathbb{E} \biggl[\int
_{0}^t \bigl\{\bigl |\chi_n(s)\bigr|^p
+ \bigl(1+\bigl |X_n(s)\bigr|^{2l} +\bigl |X_n\bigl(
\kappa_n(s)\bigr)\bigr|^{2l}\bigr)^{p/2}\\
&&\qquad\quad{}\times\bigl|X_n(s)-X_n
\bigl(\kappa_n(s)\bigr)\bigr|^p
\\
&&\qquad\quad{} +\bigl|b\bigl(s, X_n\bigl(\kappa_n(s)\bigr)
\bigr)-b_n\bigl(s,X_n\bigl(\kappa_n(s)
\bigr)\bigr)\bigr|^p
\\
& &\qquad\quad{}+ \bigl|\sigma\bigl(s,X\bigl(\kappa_n(s)\bigr)\bigr)-
\sigma_n\bigl(s,X_n\bigl(\kappa_n(s)
\bigr)\bigr)\bigr|^p \bigr\} \,ds \biggr]
\end{eqnarray*}
due to the application of Young's inequality. Note that
\[
\mathbb{E}\int_{0}^T|\chi_n(s)|^{p-2}
\chi_n(s)\alpha_n(s) \,dW(s)=0
\]
since
%
%e4.12 #&#
\begin{eqnarray}
\label{quadratic-variation}
&&\mathbb{E}\int_0^T\bigl|
\chi_n(s)\bigr|^{p-2}\bigl|\alpha^T_n(s)
\chi_n(s)\bigr| \,ds\nonumber\\
&&\qquad \le \mathbb{E}\int_0^T\bigl|
\chi_n(s)\bigr|^{p-1} \bigl(\bigl|\sigma\bigl(s,X(s)\bigr)\bigr| + \bigl|\sigma
_n\bigl(s,X_n\bigl(\kappa_n(s)\bigr)
\bigr)\bigr| \bigr) \,ds
\nonumber
\\
&&\qquad\le C \int_0^T\mathbb{E} \bigl(\bigl|
\chi_n(s)\bigr|^{p}+ \bigl|\sigma\bigl(s,X(s)\bigr)\bigr|^{p}+
\bigl|\sigma_n\bigl(s,X_n\bigl(\kappa_n(s)
\bigr)\bigr)\bigr|^p \bigr)\,ds
\nonumber
\\
&&\qquad\le C\mathbb{E}\int_0^T \bigl
\{\bigl|X(s)\bigr|^{p}+ \bigl|X_n(s)\bigr|^{p} + \bigl(1 + \bigl|X(s)\bigr|^{(l+2)}\bigr )^{p/2}
\\
&&\qquad\quad{} + \bigl(1 + \bigl|X_n\bigl(\kappa_n(s)
\bigr)\bigr|^{(l+2)} \bigr)^{p/2} \bigr\}\,ds
\nonumber
\\
&&\qquad\le C\nonumber
\end{eqnarray}
due to {B-2}, H\"{o}lder's inequality, \eqref
{diffusion_poly_growth}, Lemma~\ref{optimal_moment_bound} and that
$(l/2+1)p < p_0$ due to the $\mathfrak{p}$-\textit{condition}. Moreover,
\begin{eqnarray*}
\mathcal{E}(t)&: = & \mathbb{E}\int_0^{t} C
\bigl(1+ \bigl|X_n(s)\bigr|^{lp} + \bigl|X_n\bigl(
\kappa_n(s)\bigr)\bigr|^{lp}\bigr)\bigl|X_n(s)-X_n
\bigl(\kappa_n(s)\bigr)\bigr|^p \,ds
\\
&\le& C\int_0^{t} \bigl(\mathbb{E} \bigl[
\bigl(1+ \bigl|X_n(s)\bigr|^{lp} + \bigl|X_n\bigl(\kappa
_n(s)\bigr)\bigr|^{lp}\bigr)^{{(4l+2)}/{(3l)}} \bigr]
\bigr)^{{(3l)}/{(4l+2)}} \\
&&{}\times \bigl(\mathbb {E} \bigl[\bigl|X_n(s)-X_n
\bigl(\kappa_n(s)\bigr)\bigr|^{p{((4l+2)}/{(l+2)})} \bigr] \bigr)^{
{(l+2)}/{(4l+2)}}
\,ds
\\
&\le& C n^{-p/2}
\nonumber
\end{eqnarray*}
due to H\"{o}lder's inequality, Lemma~\ref{optimal_moment_bound} and
the fact that $p\frac{4l+2}{l+2} \le\frac{2p_0}{l +2}$ and $lp\frac
{4l+2}{3l} <\frac{4l+2}{6l+3}p_0 \le p_0$ (since it is assumed that $p
< \frac{p_0}{2l +1}$, see $\mathfrak{p}$-\textit{condition}). In view
of estimate \eqref{one-step-rate},
one deduces that
%
%e4.13 #&#
\begin{equation}
\label{rate-E} \sup_{0\le t\le T}\mathcal{E}(t) \le C
n^{-p/2}.
\end{equation}
The application of Grownwall's lemma results
in
\[
\sup_{0\le t \le T}\mathbb{E}\bigl[\bigl|\chi_n(t)\bigr|^p
\bigr]\le C n^{-p/2}
\]
due to estimate \eqref{rate-E} and Lemma~\ref{b_n_to_b}.
\end{pf*}

%s4.2 #&#
\subsection{Uniform $\mathcal{L}^p$ and a.s. convergence}

%le5 #&#
\begin{lemma} \label{Gyongy-Krylov lemma} Let $T\in[0, \infty)$ and
let $f:=\{f_t\}_{t\in[0,T]}$ and
$g:=\{g_t\}_{t\in[0,T]}$ be nonnegative continuous
$\mathbb{F}$-adapted processes such that, for any constant $c>0$,
\[
\mathbb{E}[f_{\tau}\one_{\{g_0\le c\}}] \le \mathbb{E}[g_{\tau}
\one_{\{g_0\le c\}}]
\]
for any stopping time $\tau\le T$. Then, for any stopping time
$\tau\le T$ and $\gamma\in(0,1)$,
\[
\mathbb{E}\Bigl[\sup_{t\le\tau}f_t^{\gamma}
\Bigr] \le \frac{2-\gamma}{1-\gamma} \mathbb{E}\Bigl[\sup_{t\le\tau}g_t^{\gamma}
\Bigr].
\]
\end{lemma}
\begin{pf}
See \cite{Krylovbook2} and also Gy\"{o}ngy and Krylov
\cite{KrylovIstvan2003}.
\end{pf}

\begin{pf*}{Proof of Theorem~\ref{uniform-order-thmm}}
First, fix $p$ to satisfy the $\mathfrak{p}$-\textit{condition} and
define, for every $n\ge1$, $\chi_n$, $\beta_n$ and $\alpha_n$ as in
\eqref{def_chi_beta} and \eqref{def_alpha}. Moreover, consider the
function $\phi: [0, T] \to\mathbb{R}$ which is defined by
\[
\phi(t):= \exp\bigl(-(L+2)t\bigr),
\]
where $L$ is the constant in the monotonicity condition in
{A-6}. Then It\^{o}'s formula yields
\begin{eqnarray*}
&&d\bigl(\phi(t)\bigl|\chi_n(t)\bigr|^2\bigr)^{p/2}\\
&&\qquad \le
\frac{p}{2} \phi(t)^{p/2}\bigl|\chi _n(t)\bigr|^{p-2}
\bigl(2\chi_n(t)\,d\chi_n(t) + (p-1)\bigl|
\alpha_n(t)\bigr|^2\,dt \bigr)\\
&&\qquad\quad{} - \frac{p}{2}(L+2)
\phi(t)^{p/2} \bigl|\chi_n(t)\bigr|^p\,dt
\\
&&\qquad\le \frac{p}{2} \phi(t)^{p/2}\bigl|\chi_n(t)\bigr|^{p-2}
\bigl(2\chi_n(s)\beta _n(s) + (p-1)\bigl|
\alpha_n(t)\bigr|^2 \bigr)\,dt \\
&&\qquad\quad{}- \frac{p}{2}(L+2)
\phi(t)^{p/2} \bigl|\chi_n(t)\bigr|^p \,dt
\\
&&\qquad\quad{} + p\phi(t)^{p/2}\bigl|\chi_n(t)\bigr|^{p-2}\chi
_n(s)\alpha_n(t)\,dW(t).
\end{eqnarray*}
Thus, due to \eqref{sum-difference}, one obtains that
%
%e4.14 #&#
\begin{eqnarray}
\label{integr_factor_Ito_formula} d\bigl(\phi(t)\bigl|\chi_n(t)\bigr|^2
\bigr)^{p/2} &\le& \frac{p}{2} \phi(t)^{p/2}\bigl |\chi
_n(t)\bigr|^{p-2} \bigl((L+2)\bigl|\chi_n(t)\bigr|^2
+ \eta_n(t) \bigr)\,dt\nonumber\\
&&{} - \frac
{p}{2}(L+2)\phi(t)^{p/2}
\bigl|\chi_n(t)\bigr|^p \,dt
\\
&&{} + p\phi(t)^{p/2}\bigl|\chi_n(t)\bigr|^{p-2}
\chi_n(s)\alpha_n(t)\,dW(t),\nonumber
\end{eqnarray}
where
%
%e4.15 #&#
\begin{eqnarray}
%\label{eta}
\eta_n(t)&:= & C\bigl[\bigl(1+ \bigl|X_n(s)\bigr|^{2l}
+ \bigl|X_n\bigl(\kappa_n(s)\bigr)\bigr|^{2l}\bigr)
\bigl|X_n(s)-X_n\bigl(\kappa_n(s)
\bigr)\bigr|^2
\nonumber
\\
&&{} +\bigl|b\bigl(s, X_n\bigl(\kappa_n(s)\bigr)
\bigr)-b_n\bigl(s,X_n\bigl(\kappa_n(s)
\bigr)\bigr)\bigr|^2
\\
&&{} + \bigl|\sigma\bigl(s,X\bigl(\kappa_n(s)\bigr)\bigr)-
\sigma_n\bigl(s,X_n\bigl(\kappa_n(s)
\bigr)\bigr)\bigr|^2\bigr].\nonumber
\end{eqnarray}
and $C$ is here and below a generic positive constant independent of
$n$. Consequently, one obtains for every stopping time $\tau\le T$,
due to \eqref{quadratic-variation},
\[
\mathbb{E}\bigl[\bigl(\phi(\tau)\bigl|\chi_n(\tau)\bigr|^2
\bigr)^{p/2}\bigr] \le\frac{p}{2}\mathbb {E} \biggl[\int
_0^{\tau} \bigl(\phi(t) \bigl|\chi_n(t)\bigr|^2
\bigr)^{{(p-2)}/{2}}\eta_n(t)\,dt \biggr],
\]
which results in, due to Lemma~\ref{Gyongy-Krylov lemma},
\[
\mathbb{E}\Bigl[\sup_{t\le T}\bigl(\phi(t)\bigl|\chi_n(t)\bigr|^2
\bigr)^{{p\gamma}/{2}}\Bigr] \le C \mathbb{E} \biggl[ \biggl(\int
_0^T \bigl(\phi(t)\bigl |\chi_n(t)\bigr|^2
\bigr)^{{(p-2)}/{2}}\eta_n(t)\,dt \biggr)^{\gamma} \biggr]
\]
for any $\gamma\in(0, 1)$. Then, for $p>2$, the application of
Young's inequality yields
\begin{eqnarray*}
&&\mathbb{E}\Bigl[\sup_{t\le T}\bigl(\phi(t)\bigl|\chi_n(t)\bigr|^2
\bigr)^{{p\gamma}/{2}}\Bigr]
\\
&&\qquad\le \frac{1}{2}\mathbb{E}\Bigl[\sup
_{t\le T}\bigl(\phi(t)\bigl|\chi_n(t)\bigr|^2
\bigr)^{{p\gamma
}/{2}}\Bigr] + C \mathbb{E} \biggl[ \biggl(\int
_0^T \eta_n(t)\,dt
\biggr)^{{p\gamma
}/{2}} \biggr],
\end{eqnarray*}
which implies that
\begin{eqnarray*}
\mathbb{E}\Bigl[\sup_{t\le T}\bigl(\phi(t)\bigl|\chi_n(t)\bigr|^2
\bigr)^{{p\gamma}/{2}}\Bigr] &\le& C \mathbb{E} \biggl[ \biggl(\int
_0^T \eta_n(t)^{{p}/{2}}\,dt
\biggr)^{\gamma} \biggr] \\
&\le& C \biggl( \mathbb{E} \biggl[\int
_0^T \eta_n(t)^{{p}/{2}}\,dt
\biggr] \biggr)^{\gamma}.
\end{eqnarray*}
The above estimate is also true if $p=2$, since it is an immediate
consequence of \eqref{integr_factor_Ito_formula}. Moreover, one calculates
\begin{eqnarray*}
\mathbb{E} \biggl[\int_0^T
\eta_n(t)^{{p}/{2}}\,dt \biggr]& \le& C \biggl\{ \mathcal{E}(t) +
\mathbb{E} \biggl[\int_0^T \bigl|b\bigl(s,
X_n\bigl(\kappa _n(s)\bigr)\bigr)-b_n
\bigl(s,X_n\bigl(\kappa_n(s)\bigr)
\bigr)\bigr|^p \,dt \biggr]
\\
&&{} + \mathbb{E} \biggl[\int_0^T \bigl|\sigma
\bigl(s,X\bigl(\kappa_n(s)\bigr)\bigr)-\sigma _n
\bigl(s,X_n\bigl(\kappa_n(s)\bigr)
\bigr)\bigr|^p \,dt \biggr] \biggr\}
\\
&\le& C  n^{-\alpha p}
\end{eqnarray*}
due to \eqref{rate-E}, \eqref{difference_b} and \eqref
{difference_sigma}. Thus,
\[
\mathbb{E}\Bigl[\sup_{t\le T}\bigl(\phi(t)\bigl|\chi_n(t)\bigr|^2
\bigr)^{{p\gamma}/{2}}\Bigr] \le C n^{-\alpha p \gamma},
\]
which yields the desired result
\[
\mathbb{E}\Bigl[\sup_{t\le T}\bigl|\chi_n(t)\bigr|^{p\gamma}
\Bigr] \le\exp\bigl((L+2)T\bigr) \mathbb {E}\Bigl[\sup_{t\le T}
\bigl(\phi(t)\bigl|\chi_n(t)\bigr|^2\bigr)^{{p\gamma}/{2}}\Bigr]
\le C n^{-\alpha p \gamma}.
\]
\upqed\end{pf*}

%co1 #&#
\begin{cor}
Suppose \textup{{A-2}} and \textup{{A-4}--{A-6}} hold and $p_0$ is
sufficiently large. Then the
numerical scheme \eqref{tem} with coefficients which are given by \eqref
{b2_n} and \eqref{sigma2_n} with $\alpha=1/2$ converges to the true
solution of SDE \eqref{sde} almost surely with order $\kappa<1/2$,
that is, there exists a finite random variable $\zeta_{\kappa}$ such
that almost surely
%
%e4.16 #&#
\begin{eqnarray}
\label{a.s_rate_one_half} \sup_{0 \le t \le T} \bigl|X(t)-X_n(t)\bigr| \le
\zeta_{\kappa} n^{-\kappa}
\end{eqnarray}
for any $\kappa\in(0,  \frac{1}{2}-\frac{2l+1}{p_0})$ and $l <\frac
{p_0-2}{4}$.
\end{cor}

\begin{pf}
Consider a $p \in(\frac{2}{1-2\kappa},  \frac{p_0}{2l+1})$. Then
Theorem~\ref{uniform-order-thmm} yields
\[
\mathbb{E}\Bigl[\sup_{t\le T}\bigl|X(t)-X_n(t)\bigr|^p
\Bigr] \le C n^{-p/2}.
\]
Consequently,
\begin{eqnarray*}
\sum_{n\ge1}\mathbb{P}\Bigl(\sup_{t\le T}\bigl|X(t)-X_n(t)\bigr|>n^{-\kappa}
\Bigr)&\le&\sum_{n\ge1}\mathbb{E}\Bigl[\sup
_{t\le T}\bigl|X(t)-X_n(t)\bigr|^p
\Bigr]n^{\kappa p}\\
&\le&\sum_{n\ge1} C
n^{-(1/2-\kappa)p} <\infty
\end{eqnarray*}
and thus, the Borel--Cantelli lemma implies that there exits a finite
random variable $\zeta_{\kappa}$ such that almost surely
\[
\sup_{t\le T}\bigl|X(t)-X_n(t)\bigr|\le\zeta_{\kappa}
n^{-\kappa}.
\]
\upqed\end{pf}

%
%s5 #&#
\section{Simulation results}

%t1 #&#
\begin{table}
\tablewidth=200pt
\caption{Errors in the tamed Euler scheme}\label{ex1:table1}
\begin{tabular*}{200pt}{@{\extracolsep{\fill}}lc@{}}
\hline
\textbf{Step-size}& $\bolds{\sqrt{E|X(t)-X_n(t)|^2}}$ \\
\hline
$2^{-19}$& 0.0007546660690748 \\
$2^{-18}$& 0.0014293698755019 \\
$2^{-17}$& 0.0024054188924763 \\
$2^{-16}$& 0.0036583313232057 \\
$2^{-15}$& 0.0053921530728755 \\
$2^{-14}$& 0.0080671890795787 \\
$2^{-13}$& 0.0118014601267312 \\
$2^{-12}$& 0.0165751338687870 \\
$2^{-11}$& 0.0236798743828524 \\
$2^{-10}$& 0.0322254347247282 \\
$2^{-09}$& 0.0445565040073459 \\
$2^{-08}$& 0.0614016271396012 \\
$2^{-07}$& 0.0826347082207412 \\
$2^{-06}$& 0.1085948479470830 \\
\hline
\end{tabular*}
\end{table}

In order to further support the theoretical results obtained in this
article, simulation results are presented for the following nonlinear
(2-$d$) stochastic differential equation (see also Section~\ref{Intro}
for comparison with \cite{HJ} and \cite{Tretyakov-Zhang}),
\[
dX(t)=\lambda X(t) \bigl(\mu-\bigl|X(t)\bigr|\bigr)\,dt +
 \xi\bigl|X(t)\bigr|^{3/2}\,dW_t,
\]
where the initial data $X_0= [1,   1]^T$, $\lambda=2.5$, $\mu= 1$,
$\xi$ is the following positive definite matrix:
\[
\pmatrix{ \frac{2}{\sqrt{10}} &
\frac{1}{\sqrt{10}}
\vspace*{2pt}\cr
\frac{1}{\sqrt{10}} & \frac{2}{\sqrt{10}} }
\]
with $|\xi|=1$, and $T=1$. The outputs in Table~\ref{ex1:table1} and
Figure~\ref{ex1:figure1}\footnote
{Table~\ref{ex1:table1} and Figure~\ref{ex1:figure1} are courtesy of Chaman Kumar.} are based on 1000
simulations, that is, simulated paths, of scheme \eqref{tem} with
coefficients given by Model 2, that is, \eqref{b2_n} and \eqref
{sigma2_n} with $l=1$, and presented by using the $\log_2$ scale.

%f1 #&#
\begin{figure}

\includegraphics{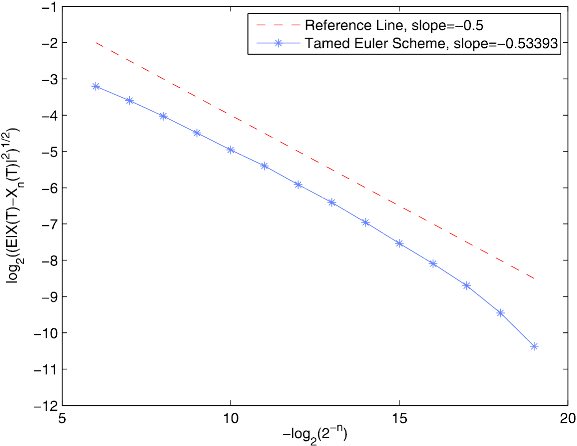}

\caption{Rate of convergence of the explicit Euler scheme (Model 2).}
 \label{ex1:figure1}
\end{figure}

%\renewcommand{\theequation}{A-\arabic{equation}}
% redefine the command that creates the equation no.
%\setcounter{equation}{0} % reset counter

\begin{appendix}\label{app}
\section*{Appendix} % use *-form to suppress numbering
Consider the following $d$-dimensional SDE which is given by
\[
dX_t=\lambda X_t\bigl(\mu-|X_t|\bigr)\,dt +
\xi|X_t|^{3/2}\,dW_t\qquad \forall t\in[0, T],
\]
with initial condition $X_0\in\mathbb R^d$, where $\lambda$, $\mu$ and
all elements of the vector $X_0$ are positive constants. Moreover, $\xi
\in \mathbb R^{d\times d_1}$ is a positive definite matrix and $\{
W(t)\}_{t \geq0}$ is a $d_1$-dimensional Wiener martingale. One then
defines $b(x): = \lambda x (\mu-|x|)$ and $\sigma(x):= \xi|x|^{3/2}$
for every $x\in\mathbb R^d$ and observes that the coercivity condition
{A-4}
\[
2xb(x) + (p_0-1)\bigl|\sigma(x)\bigr|^2 \leq K
\bigl(1+|x|^2\bigr),
\]
is satisfied with $p_0 \le\frac{2\lambda+|\xi|^2}{|\xi|^2}$ and
$K=2\lambda\mu$ for all $x, y \in\mathbb{R}^d$. Moreover, one obtains that
\[
(x-y)\bigl[b(x)-b(y)\bigr]  \le \lambda\mu|x-y|^2-\lambda\bigl(|x|+|y|\bigr)
\bigl(|x|-|y|\bigr)^2
\]
and
\[
\bigl|\sigma(x)-\sigma(y)\bigr|^2  \le2 |\xi|^2\bigl(|x|+|y|\bigr)
\bigl(|x|-|y|\bigr)^2.
\]
As a result, the monotonicity condition in {A-6}
\[
2(x-y) \bigl(b(t,x)-b(t,y)\bigr) + (p_1-1)\bigl|\sigma(t,x)-
\sigma(t,y)\bigr|^2 \leq L |x-y|^2
\]
is satisfied with $p_1 \le\frac{\lambda+|\xi|^2}{|\xi|^2}$ and
$L=2\lambda\mu$ for all $x, y \in\mathbb{R}^d$. Finally, one easily
obtains that
\[
\bigl|b(x)-b(y)\bigr| \le\lambda\max(\mu,1) \bigl(1+|x|+|y|\bigr)|x-y|
\qquad\mbox{for all } x, y \in
\mathbb{R}^d,
\]
to conclude that $l=1$ in {A-6}.
\end{appendix}
%\begin{appendix}
%\section{}
%\end{appendix}

%\section*{Acknowledgments}

%\begin{supplement}[id=suppA]
%\slink[doi]{10.1214/00-AAPXXXXSUPP} %[doi,text={...}] - jei reikia
%suskaldyti doi
%\sfilename{aapXXXX\_supp.pdf}
%\end{supplement}

%\begin{thebibliography}{99}
%\bibitem[\protect\citeauthoryear{}{}]{r1}
%\bibitem{r1}
%\end{thebibliography}

\printaddresses
\end{document}